\documentclass[a4paper,12pt]{amsart}

\usepackage{amsmath, amssymb, amsthm}
\usepackage{hyperref}
\usepackage{xypic}
\usepackage{graphicx}

\usepackage{verbatim}
\usepackage{enumitem}
\allowdisplaybreaks

\newcounter{theorem}
\newtheorem{thm}[theorem]{Theorem}
\newtheorem{lemma}[theorem]{Lemma}
\newtheorem{prop}[theorem]{Proposition}
\newtheorem{cor}[theorem]{Corollary}
\newtheorem{defn}[theorem]{Definition}

\newtheorem{thmx}{Theorem}

\theoremstyle{remark}
\newtheorem*{remark*}{Remark}
\newtheorem{remark}[theorem]{Remark}
\newtheorem{example}[theorem]{Example}

\numberwithin{equation}{section}
\numberwithin{theorem}{section}

\newcommand{\e}{\epsilon}
\newcommand{\dl}{\delta}

\newcommand{\N}{\mathbb{N}}
\newcommand{\cn}{\mathcal{N}}
\newcommand{\unit}{\mathbf{1}}

\renewcommand{\setminus}{\backslash}

\newcommand{\centredrelation}[2]{{}_{\phantom{#2}}#1_{#2}}
\newcommand{\capprox}[1]{\centredrelation{\approx}{#1}}

\begin{document}

\title{Almost Finiteness, comparison, and Tracial $\mathcal{Z}$-Stability}

\author{Hung-Chang Liao}
\address{Department of Mathematics and Statistics, University of Ottawa,
	\newline 150 Louis-Pasteur Pvt, Ottawa, ON, Canada K1N 6N5}
\email{hliao@uottawa.ca}

\author{Aaron Tikuisis}
\address{Department of Mathematics and Statistics, University of Ottawa,
	\newline 150 Louis-Pasteur Pvt, Ottawa, ON, Canada K1N 6N5}
\email{aaron.tikuisis@uottawa.ca}

\maketitle

\begin{abstract} Inspired by Kerr's work on topological dynamics, we define tracial $\mathcal{Z}$-stability for sub-$C^*$-algebras. We prove that for a countable discrete amenable group $G$ acting freely and minimally on a compact metrizable space $X$, tracial $\mathcal{Z}$-stability for the sub-$C^*$-algebra $(C(X)\subseteq C(X)\rtimes G)$ implies that the action has dynamical comparison. Consequently, tracial $\mathcal{Z}$-stability is equivalent to almost finiteness of the action, provided that the action has the small boundary property. 
\end{abstract}

\section*{Introduction}
\renewcommand{\thethmx}{\Alph{thmx}}
\renewcommand{\thecorx}{\Alph{corx}}

Operator algebras and dynamical systems have long enjoyed a rich and sometimes surprising relationship. Murray and von Neumann's uniqueness theorem on hyperfinite II$_1$ factors (\cite{MvN}) is paralleled by Dye's uniqueness of hyperfinite ergodic measure-preserving equivalence relations (\cite{Dye}). Connes' celebrated ``injectivity implies hyperfiniteness'' theorem (\cite{Connes}) also has a direct analogue --- the Connes--Feldman--Weiss theorem --- which asserts that amenable nonsingular measured equivalence relations are hyperfinite (\cite{CFW}). The connection between these two fields goes far beyond mere analogies.
There is a natural construction of a von Neumann algebra from a measure-preserving equivalence relation, and amenable (respectively, hyperfinite) equivalence relations are exactly those whose associated von Neumann algebras are amenable (respectively, hyperfinite). In fact, Feldman and Moore showed in \cite{FM2} that there is a full-fledged correspondence between Cartan subalgebras of von Neumann algebras and (twisted) measured equivalence relations.

The same types of analogies and connections are emerging between $C^*$-algebras and topological dynamics.
Here the interest largely surrounds properties in the Toms--Winter conjecture --- a $C^*$-algebraic conjecture designed to create a robust notion of regularity, primarily for the purpose of identifying classifiable $C^*$-algebras.\footnote{Amenability is also of interest, but developments on both the $C^*$- and dynamical sides (e.g., \cite{Villadsen,GiolKerr}) show that amenability alone has weaknesses in terms of classifying $C^*$-algebras. Hence the question here is: under the base assumption of amenability, what additional conditions are needed to ensure regularity?} More precisely, the Toms--Winter conjecture predicts that for simple unital separable non-elementary nuclear $C^*$-algebras, the following are equivalent:
\begin{enumerate}
	\item[(C1)] finite nuclear dimension;
	\item[(C2)] $\mathcal{Z}$-stability;
	\item[(C3)] strict comparison.
\end{enumerate}
The implications (C1) $\Rightarrow$ (C2) and (C2) $\Rightarrow$ (C3) were established in full generality in \cite{Winter12} and \cite{Rordam:IJM}, respectively. (C2) $\Rightarrow$ (C1) was recently settled in \cite{CETW}, building on the groundbreaking work of Matui and Sato (\cite{MatuiSato:Duke}). The only implication which is not completely settled is (C3) $\Rightarrow$ (C2). The most general result at the moment was obtained in \cite{CETW} (based on another fundamental work of Matui and Sato, \cite{MS:Acta}): strict comparison implies $\mathcal{Z}$-stability provided that the $C^*$-algebra has ``uniform property $\Gamma$''.

On the dynamical side, Kerr defined for actions of countable discrete amenable groups on compact metrizable spaces the following properties (\cite{Kerr:JEMS}), which are analogous to the three conditions in the Toms--Winter conjecture:
\begin{enumerate}
	\item[(D1)] finite tower dimension;
	\item[(D2)] almost finiteness (this partially generalizes Matui's almost finiteness for topological groupoids with totally disconnected unit spaces);
	\item[(D3)] dynamical comparison (comparison for topological dynamics first appeared in talks of Wilhelm Winter).
\end{enumerate}
Kerr proved that (among other things) for free minimal actions
\begin{enumerate}
	\item[(1)] almost finiteness implies dynamical comparison;
	\item[(2)] if the action is almost finite, then the crossed product is (tracially) $\mathcal{Z}$-stable.
\end{enumerate}
In \cite{KerrSzabo}, Kerr and Szab\'o showed a relationship between almost finiteness and dynamical comparison which is reminiscent of the relationship between $\mathcal{Z}$-stability and strict comparison above. More precisely, they showed that dynamical comparison implies almost finiteness provided that the action has the small boundary property in the sense of \cite{LindenstraussWeiss}.

Inspired by the apparent analogy between almost finiteness and $\mathcal Z$-stability, we sought to establish a formal link here, provided the Cartan structure is taken into account on the $C^*$-algebraic side.
We took a McDuff-type characterization of $\mathcal Z$-stability for nuclear $C^*$-algebras due to Hirshberg and Orovitz (called tracial $\mathcal Z$-stability, \cite{HirshbergOrovitz}, and their characterization makes key use of ideas of Matui and Sato, \cite{MS:Acta}), and strengthened it in ways naturally related to the Cartan structure, to produce a property we call \emph{tracial $\mathcal Z$-stability} for a sub-$C^*$-algebra (Definition \ref{defn:Zstable}). Our definition is, largely, inspired by a close analysis of Kerr's proof that minimal free actions of amenable groups which are almost finite give rise to $\mathcal{Z}$-stable crossed products. Our main result shows that this definition is precisely how the Cartan structure encodes almost finiteness, at least at the presence of the small boundary property.

\begin{thmx} [Corollary \ref{cor:AlmostFiniteZstable}] \label{thm:A}
	Let $G$ be a countable discrete amenable group, let $X$ be a compact metrizable space, and let $\alpha:G \curvearrowright X$ be a free minimal action. Consider the following conditions:
	\begin{enumerate}
		\item $\alpha$ is almost finite;
		\item $(C(X) \subseteq C(X)\rtimes_\alpha G )$ is tracially $\mathcal Z$-stable;
		\item $\alpha$ has dynamical comparison.
	\end{enumerate}
	Then $(i)\Rightarrow (ii) \Rightarrow (iii)$, and if $\alpha$ has the small boundary property then all three conditions are equivalent.
\end{thmx}

The key novelty in our definition lies in the use of ``one-sided normalizers''. The study of (two-sided) normalizers has a long history in operator algebras, and they play a fundamental role in the connection between dynamics and operator algebras. Indeed, Feldman and Moore proved in \cite{FM2} that one can reconstruct an equivalence relation from its associated Cartan subalgebra and normalizers (similar results in the $C^*$-setting were obtained in \cite{Kumjian:CJM} and \cite{Renault:IMSB}). In this paper we show that one-sided normalizers also carry a significant amount of dynamical information. In particular, the dynamical subequivalence that appears in the definition of dynamical comparison can be completely characterized using one-sided normalizers (Proposition \ref{prop:DynamicalCuntzCharacterization}).

The paper is organized as follows. In Section \ref{sec:Prelim}, we establish our notation and record several facts about sub-$C^*$-algebras and normalizers.
In Section \ref{sec:Preorder} we study one-sided normalizers and prove a $C^*$-characterization of the dynamical subequivalence defined in \cite{Kerr:JEMS}  using these normalizers. (As mentioned in \cite{Kerr:JEMS}, this subequivalence relation first appeared in talks of Winter.) In Section \ref{sec:Z} we prove Theorem \ref{thm:A}  (as Corollary \ref{cor:AlmostFiniteZstable}).

\subsection*{Acknowledgements}
This research is supported by an NSERC Discovery Grant.
H.L.\ is also supported by the Fields Institute.

\section{Preliminaries}
\label{sec:Prelim}

For a $C^*$-algebra $A$ we write $A_+$ for the set of positive elements in $A$, $A^1$ for the set of elements of norm at most $1$, and $A_+^1$ for the intersection (the set of positive contractions).
We write $T(A)$ for the set of tracial states on $A$. For two elements $a,b$ in $A$ and $\eta > 0$ we write $a\approx_\eta b$ if $\| a -b\| < \eta$.

\begin{defn}
	A \emph{sub-$C^*$-algebra} $(D\subseteq A)$ refers to a $C^*$-algebra $A$ together with a $C^*$-subalgebra $D$. We say a sub-$C^*$-algebra $(D\subseteq A)$ is \emph{nondegenerate} if $D$ contains an approximate unit for $A$.
\end{defn}

\begin{defn} \label{defn:normalizer}
	Let $(D\subseteq A)$ be a sub-$C^*$-algebra. An element $a\in A$ is said to \emph{normalize} $D$ if $a^*Da + aDa^* \subseteq D$ (we also say $a$ is a \emph{normalizer} of $D$ in $A$). The set of normalizers of $D$ in $A$ is denoted by $\cn_A(D)$.
\end{defn}

It follows directly from the definition that the set of normalizers is closed under multiplication, involution, and norm-limits.

Although we won't explicitly need the definition, we recall the notion of a ($C^*$-algebra) Cartan subalgebra, as this is main context to keep in mind when we work with sub-$C^*$-algebras.
A \emph{Cartan subalgebra} is a nondegenerate sub-$C^*$-algebra $(D\subseteq A)$ where $D$ is a maximal abelian subalgebra, such that there exists a faithful conditonal expectation $E:A \to D$, and such that $\mathcal N_A(D)$ generates $A$ as a $C^*$-algebra.
 
It is useful to know that, in many cases, if $a$ is a normalizer of $D$ then $a^*a,aa^*$ belong to $D$.
This is not true in general (for example, take $D = \{0\}$), but is under the assumption of nondegeneracy, as the following shows.

\begin{lemma}[{\cite[Lemma 4.6]{Renault:IMSB}}] \label{lem:NondegNormalizer}
	Let $(D\subseteq A)$ be a nondegenerate sub-$C^*$-algebra and $a\in \cn_A(D)$. Then $a^*a$ and $aa^*$ belong to $D$.
\end{lemma}
\begin{proof}
	Let $(u_\lambda)_\lambda$ be an approximate unit in $D$ for $A$. Then by definition $a^*u_\lambda a$ is in $D$ for every $\lambda$. Since $\cn_A(D)$ is closed under norm-limits, we see that $a^*a\in D$. The same argument shows that $aa^*$ also belongs to $D$.
\end{proof}

Although a sum of normalizers is not necessarily a normalizer, it will be if the subalgebra is abelian and a certain orthogonality condition is satisfied.

\begin{lemma} \label{lem:SumOfNormalizers}
	Let $(D\subseteq A)$ be a sub-$C^*$-algebra with $D$ abelian, and $x_1,...,x_n$ be  normalizers of $D$ in $A$. Set $z = \sum_{i=1}^n x_i$.
	\begin{enumerate}
		\item If $x_i^*x_j = 0$ whenever $i \neq j$, then $z^*Dz \subseteq D$. 
		\item If $x_ix_j^* = 0$ whenever $i \neq j$, then $zDz^* \subseteq D$. 
	\end{enumerate}
	As a consequence, if $x_i^*x_j = 0 = x_ix_j^*$ whenever $i\neq j$, then $z$ belongs to $\cn_A(D)$.
\end{lemma}   
\begin{proof}
	We only prove (i) as the other assertion is completely analogous. It suffices to show that if $i\neq j$ then $x_i^*dx_j = 0$ for every $d\in D$. For this observe that
	\begin{align*}
	\| x_i^*d x_j\|^4 &= \| x_j^*d^*x_ix_i^* d x_j \|^2 = \| x_j^*d^* x_ix_i^* d x_j	x_j^*d^*x_ix_i^* d x_j \| \\
	&\leq  \| x_j\|^2 \| x_j^*d^* x_ix_i^* dd^*x_ix_i^* d x_j \|.
	\end{align*}
	Note that 
	\[  x_ix_i^* dd^*x_ix_i^* = x_i(x_i^* dd^*x_i)x_i^* \in x_iDx_i^* \subseteq D.  \]
	Since $D$ is abelian,
	\[ x_j^*d^* (x_ix_i^* dd^*x_ix_i^*) d x_j = x_j^*d^*d( x_ix_i^* dd^*x_ix_i^* )x_j  = 0. \]
	It follows that $x_i^*dx_j = 0$ and the proof is complete.
\end{proof}

Throughout the paper we write $M_n$ for the algebra of all complex-valued $n$-by-$n$ matrices, and $D_n$ for the subalgebra of $M_n$ that consists of all diagonal matrices.

\begin{example}[{\cite[Example 2]{Kumjian:CJM}}]\label{exam:MatrixNormalizer}
	An element $a\in M_n$ belongs to $\cn_{M_n}(D_n)$ if and only if $a$ has at most one nonzero entry in each row and each column.
\end{example}

In this paper, normalizer-preserving maps between sub-$C^*$-al\-ge\-bras will play an essential role. 

\begin{lemma} \label{lem:NormalizersPreserving}
	Let $( D_A \subseteq A )$ and $(D_B \subseteq B)$ be two sub-$C^*$-algebras and $\phi:A\to B$ be a positive linear map. Suppose $(D_B\subseteq B)$ is nondegenerate.
	\begin{enumerate}
		\item If $\phi( \cn_A(D_A)  )\subseteq \cn_B(D_B)$, then $\phi(D_A)\subseteq D_B$.
		\item Suppose in addition that $(D_A\subseteq A)$ is nondegenerate and $\phi$ is a $^*$-homomorphism. If $\phi(\cn_A(D_A)) = \cn_B(D_B)$, then $\phi(D_A) = D_B$.
	\end{enumerate}
\end{lemma}
\begin{proof}
	\begin{enumerate}
		\item Let $d$ be a positive element in $D_A$. Then $\phi(d)$ belongs to $\cn_B(D_B)$ because $D_A$ is contained in $\cn_A(D_A)$. Since $\phi$ is positive, $\phi(d)^2 = \phi(d)^*\phi(d)$ and the later belongs to $D_B$ by Lemma \ref{lem:NondegNormalizer}.  It follows that $\phi(d)\in D_B$ and from linearity we conclude that $\phi(D_A)\subseteq D_B$.
		\item By (i) we only need to show that $D_B\subseteq \phi(D_A)$, so let $e$ be a positive element in $D_B$. Since $D_B\subseteq \cn_B(D_B)$, by assumption we can find an element $a\in \cn_A(D_A)$ such that $\phi(a) = e$. Using the fact that $\phi$ is a $^*$-homomorphism, we see that
		\[ \phi(|a|) = \phi( (a^*a)^{ \frac{1}{2} } ) =  (\phi(a)^*\phi(a  ))^{\frac{1}{2}} = (e^2)^{\frac{1}{2} } = e.    \] As $(D_A\subseteq A)$ is nondegenerate, $|a|$ belongs to $D_A$ by Lemma \ref{lem:NondegNormalizer} and the proof is complete.
	\end{enumerate}
\end{proof}

\begin{remark}
	In the previous lemma nondegeneracy is necessary in both assertions. For example, if we take $D_B = \{0\}$ then any map from $A$ into $B$ is normalizer-preserving (because $\cn_B(D_B) = B$) but in most cases it does not map $D_A$ into $D_B$. For (ii) we can take $D_A = \{0\}$ and $D_B = B$. Then any surjective $^*$-homomorphism from $A$ onto $B$ maps $\cn_A(D_A)$ (which is all of $A$) onto $\cn_B(D_B)$ (which is all of $B$), but the image of $D_A$ is $\{0\}$.
\end{remark}

Recall that a c.p.c.\ map $\phi:A\to B$ be two $C^*$-algebras is \emph{order zero} if $\phi(a)\phi(b) = 0$ whenever $a,b$ are positive elements in $A$ satisfying $ab = 0$. Using the structure theorem (\cite[Theorem 3.3]{WZ:Muenster}) we see that order zero maps in fact preserve arbitrary orthogonality: $\phi(a)\phi(b) = 0$ if $ab = 0$. The next observation follows from Example \ref{exam:MatrixNormalizer} and Lemma \ref{lem:SumOfNormalizers}.

\begin{lemma} \label{lem:OrderZeroMatrixUnit}
	Let $(D\subseteq A)$ be a nondegenerate sub-$C^*$-algebra,  $\phi:M_n\to A$ be a c.p.c.\ order zero map. Then $\phi( \cn_{M_n}(D_n)  )\subseteq \cn_A(D)$ if and only if $\phi(e_{ij})\in \cn_A(D)$ for every matrix unit $e_{ij}$.
\end{lemma}

\section{$C^*$-algebraic characterization of dynamical subequivalence}
\label{sec:Preorder}

Given a group acting on a compact Hausdorff space, one obtains a preorder on the open sets of the space which encodes certain information about the dynamics.

\begin{defn}[{\cite[Definition 3.1]{Kerr:JEMS}}]\label{defn:DynamicalPreorder} 
Let $G$ be a countable discrete group, let $X$ be a compact Hausdorff space, and let $\alpha:G\curvearrowright X$ be an action by homeomorphisms. For a closed set $F$ and an open set $V$ in $X$, we write $F\prec V$ if there are open sets $U_1,...,U_n$ and $s_1,...,s_n\in G$ such that $F\subseteq \bigcup_{i=1}^n U_i$ and that $s_iU_i$ are pairwise disjoint subsets of $V$.

For open sets $O, V$ in $X$, we write $O\prec V$ if $F\prec V$ for every closed subset $F$ of $O$.
\end{defn}

This preorder was first defined in talks of Wilhelm Winter, and features prominently in \cite{Kerr:JEMS,KerrSzabo}.

In \cite{Ma:ETDS}, this definition was extended to tuples of open sets, which we recall below. The main motivation was to study a generalized version of the classical type semigroup, and use it to characterize dynamical comparison.

\newcommand{\suppo}{\mathrm{supp}^\circ}
For a continuous function $f \in C_0(X)$ on a locally compact Hausdorff space $X$, we denote
\begin{equation}\label{eq:suppo} \suppo(f):=f^{-1}(\mathbb C \setminus \{0\})   \end{equation} 
(the \emph{open support} of $f$).

\newcommand{\diag}{\mathrm{diag}}
\begin{defn}[{\cite[Definitions 1.4 and 2.1]{Ma:ETDS}}]
\label{defn:DynamicalCuntz}
Let $G$ be a countable discrete group, let $X$ be a compact Hausdorff space, and let $\alpha:G\curvearrowright X$ be an action by homeomorphisms.
For a tuple of compact sets $F_1,\dots,F_n \subseteq X$ and a tuple of open sets $V_1,\dots,V_m \subseteq X$, write $(F_1,\dots,F_n)\prec (V_1,\dots,V_m)$ if
there are open sets $U_{i,j}\subseteq X$, group elements $s_{i,j} \in G$, and indices $k_{i,j} \in \{1,\dots,m\}$ for $i=1,\dots,n$ and $j=1,\dots,J_i$ such that:
\begin{enumerate}
\item For each $i$,
\begin{equation} F_i \subseteq U_{1i,1}\cup \cdots \cup U_{i,J_i}, \end{equation}
\item
\begin{equation} \coprod_{i=1}^n \coprod_{j=1}^{J_i} s_{i,j}U_{i,j} \times \{k_{i,j}\} \subseteq \coprod_{l=1}^m V_l \times \{l\}. \end{equation}
\end{enumerate}

For $a=\diag(a_1,\dots,a_n) \in (D_n \otimes C(X))_+$ and $b=\diag(b_1,\dots,b_m)\in (D_m\otimes C(X))_+$, we write $a \preccurlyeq b$ if $(F_1,\dots,F_n)\prec (\suppo(b_1),\dots,\suppo(b_m))$ for every tuple of compact sets $F_1,\dots,F_n$ with $F_i \subseteq \suppo(a_i)$ for all $i$.
\end{defn}

It was noted in \cite{Ma:ETDS} that the preorder $\preccurlyeq$ is closely related to the Cuntz subequivalence. Indeed, \cite[Proposition 2.3]{Ma:ETDS} shows that for $a = \diag(a_1,...,a_n)\in (D_n \otimes C(X))_+$ and $b=\diag(b_1,\dots,b_m)\in (D_m\otimes C(X))_+$, if $a\preccurlyeq b$ then $a$ is Cuntz subequivalent to $b$ in $C(X)\rtimes_\alpha G$. The main result of this section shows that this preorder is completely characterized by a refined Cuntz subequivalence that makes use of ``one-sided normalizers".

\begin{defn}
Let $(D\subseteq A)$ be a sub-$C^*$-algebra.
An element $a \in A$ is an \emph{$r$-normalizer} of $D$ if $a^*Da \subseteq D$.
It is an \emph{$s$-normalizer} of $D$ if $aDa^* \subseteq D$.

The set of $r$-normalizers of $D$ in $A$ is denoted $\mathcal{RN}_A(D)$, and the set of $s$-normalizers of $D$ in $A$ is denoted $\mathcal{SN}_A(D)$.
\end{defn}

The names ``$r$-normalizer'' and ``$s$-normalizer'' are motivated by a connection to $r$- and $s$-sections for groupoids, established in Proposition \ref{prop:SeminormalizerSection} below.
The following facts are evident:
\begin{itemize}
\item The product of two $r$-normalizers is an $r$-normalizer, and likewise for $s$-normalizers;
\item $\mathcal{RN}_A(D)=\mathcal{SN}_A(D)^*$;
\item an element is a normalizer if and only if it is both an $r$- and an $s$-normalizer.
\end{itemize}

Here is a useful descriptions of $r$-normalizers in matrix amplifications.

\begin{lemma}
\label{lem:MatrixSeminormalizer}
Let $(D
\subseteq A)$ be a sub-$C^*$-algebra and let
\begin{equation} x = \left(\begin{array}{ccc} x_{11} & \cdots & x_{1n} \\ \vdots && \vdots \\ x_{n1} & \cdots & x_{nn} \end{array}\right) \in M_n\otimes A. \end{equation}
Then $x \in \mathcal{RN}_{M_n\otimes A}(D_n\otimes D)$ if and only if
\begin{enumerate}
\item $x_{ij} \in \mathcal{RN}_{A}(D)$ for all $i,j$, and
\item for all $i,j,k$ with $i\neq j$ and all $a \in D$, $x_{ki}^*ax_{kj}=0$.
\end{enumerate}
\end{lemma}

\begin{proof}
For $b=\diag(b_1,\dots,b_n) \in D_n\otimes D$, we compute that the $(i,j)$-entry of $x^*bx$ is
\begin{equation} \sum_{k=1}^n x_{ki}^*b_kv_{kj}. \end{equation}

Suppose that $x \in \mathcal{RN}_{M_n\otimes A}(D_n\otimes D)$, so this must always be in $D$, and it must moreover be $0$ whenever $i\neq j$.
By setting $b_k:=a$ and $b_l:=0$ for $l\neq k$, we thus get that $x_{ki}^*ax_{kj} \in D$, and is moreover $0$ if $i\neq j$.
This shows both (i) and (ii).

Conversely, suppose that (i) and (ii) hold.
By (ii), it follows that the $(i,j)$-entry of $x^*bx$ is $0$ whenever $i\neq j$, i.e., $x^*bx$ is a diagonal matrix.
Moreover by (i), it follows that $x_{ki}^*b_kx_{ki}\in D$ for all $i,k$, and thus $x^*bx \in D_n\otimes D$.
This shows that $x \in \mathcal{RN}_{M_n\otimes A}(D_n\otimes D)$.
\end{proof}

These one-sided normalizers are best understood in the context of groupoids. For a locally compact Hausdorff \'etale groupoid $\mathcal{G}$ we write $\mathcal{G}^{(0)}$ for its unit space, $r$ and $s$ for the range and source map, respectively. If $x$ is a point in $\mathcal G^{(0)}$ then we write  $\mathcal G_x := \{ \gamma\in \mathcal G: s(\gamma ) = x  \}$ and $\mathcal G^x := \{ \gamma \in \mathcal G: r(\gamma ) = x  \}$.  We refer the readers to \cite{Sims:Notes} for more on \'etale groupoids and their $C^*$-algebras.

\begin{defn}[{\cite[Section 3]{Renault:IMSB}}]
Let $\mathcal G$ be a locally compact Hausdorff \'etale groupoid.
A subset $A$ of $\mathcal G$ is an \emph{$r$-section} if $r|_A:A \to \mathcal G^{(0)}$ is injective; it is an \emph{$s$-section} if $s|_A:A \to \mathcal G^{(0)}$ is injective.
\end{defn}

We note for context that the more familiar notion of a \emph{bisection} is a subset $A \subseteq \mathcal G$ which is both an $r$-section and an $s$-section.

It is known that when the groupoid is topologically principal, normalizers are exactly functions that are supported in bisections (\cite[Proposition 4.8]{Renault:IMSB}; in the case of principal \'etale groupoids, see \cite[Proposition 1.6]{Kumjian:CJM}), a fact classically proven using the polar decomposition of a normalizer.
We generalize this fact, at least in the principal case, to $r$-normalizers and $r$-sections; however, we give a completely different argument, since the partial isometry in the polar decomposition of an $r$-normalizer no longer leads to an $r$-section.

In the following, we make sense of the open support of an element $a \in C^*_r(\mathcal G)$, as a subset of $\mathcal G$, via the canonical (non-algebraic, non-isometric) embedding $C^*_r(\mathcal G) \to C_0(\mathcal G)$ (see, for example, \cite[Proposition 3.3.3]{Sims:Notes}).

\begin{prop}
\label{prop:SeminormalizerSection}
Let $\mathcal G$ be a locally compact Hausdorff principal \'etale groupoid, and let $a \in C^*_r(\mathcal G)$.
Then $a \in \mathcal{RN}_{C^*_r(\mathcal G)}(C_0(\mathcal G^{(0)}))$ if and only if $\suppo(a)$ is an $r$-section.
Likewise, $a \in \mathcal{SN}_{C^*_r(\mathcal G)}(C_0(\mathcal G^{(0)}))$ if and only if $\suppo(a)$ is an $s$-section.
\end{prop}

\begin{proof}
Using the adjoint, the second statement is equivalent to the first, which is the one we'll prove.
Set $A:=\suppo(a)$.

For $f \in C_0(\mathcal G^{(0)})$ and $\gamma \in \mathcal G$, we compute
\begin{align}
\notag
(a^*fa)(\gamma)
&= \sum_{s(\alpha)=r(\gamma)} a^*(\alpha^{-1})f(r(\alpha))a(\alpha\gamma) \\
&= \sum_{s(\alpha)=r(\gamma)} \overline{a(\alpha)}f(r(\alpha))a(\alpha\gamma),
\label{eq:SeminormalizerSection1}
\end{align}
and note that the summand can only be nonzero when both $\alpha$ and $\alpha\gamma$ are in $A$.

Thus, if $A$ is an $r$-section, then nonzero summands can only arise if $\gamma$ is a unit, so that $a^*fa \in C_0(\mathcal G^{(0)})$. 

For the other direction, suppose for a contradiction that there exist distinct elements $\gamma_1,\gamma_2 \in A$ such that $r(\gamma_1)=r(\gamma_2)$; then we set
\begin{equation} \gamma:=\gamma_1^{-1}\gamma_2 \in \mathcal G \setminus \mathcal G^{(0)}. \end{equation}
By \cite[Proposition II.4.1 (i)]{Renault:Book}, the sums $\sum_{s(\alpha)=r(\gamma)} |a(\alpha)|^2$ and $\sum_{s(\alpha)=r(\gamma)} |a(\alpha\gamma)|^2$ converge; thus by the Cauchy--Schwarz inequality, so does
\begin{equation} \sum_{s(\alpha)=r(\gamma)} |a(\alpha)a(\alpha\gamma)|. \end{equation}
Therefore we may find a finite set $F$ of $\mathcal G_{r(\gamma)}$ such that
\begin{equation}
\label{eq:SeminormalizerSection2}
 \sum_{\alpha \in \mathcal G_{r(\gamma)}\setminus F} |a(\alpha)a(\alpha\gamma)| < |a(\gamma_1)a(\gamma_2)|.
\end{equation}
Choose a function $f \in C_0(\mathcal G^{(0)})$ of norm $1$ such that $f(r(\gamma_1))=1$ and $f(r(\alpha))=0$ for $\alpha \in F \setminus \{\gamma_1\}$.
(Since $\mathcal G$ is principal, $r(\gamma_1) \neq r(\alpha)$ for any $\alpha \in F \setminus \{\gamma_1\}$, so this is possible.)
Then
\begin{eqnarray}
\notag
|(a^*fa)(\gamma)| 
&\stackrel{\eqref{eq:SeminormalizerSection1}}=& \big| \sum_{s(\alpha)=r(\gamma)} \overline{a(\alpha)}f(r(\alpha))a(\alpha\gamma) \big| \\
\notag
&\stackrel{\eqref{eq:SeminormalizerSection2}}>& \big|\sum_{\alpha \in F} \overline{a(\alpha)}f(r(\alpha))a(\alpha\gamma) \big| - |a(\gamma_1)a(\gamma_2)| \\
\notag
&=& |\overline{a(\gamma_1)}f(r(\gamma_1))a(\gamma_1\gamma)| -|a(\gamma_1)a(\gamma_2)| \\
&=& 0.
\end{eqnarray}
Since $\gamma \not\in \mathcal G^{(0)}$, this implies that $a^*fa \not\in C_0(\mathcal G^{(0)})$, which contradicts the hypothesis that $a \in \mathcal{RN}_{C^*_r(\mathcal G)}(C_0(\mathcal G^{(0)}))$.
\end{proof}

Specializing the above to the case of interest in this paper -- that $\mathcal G$ is a transformation groupoid $G \times X$ -- yields the next corollary.
In the following, for an element $a$ of the crossed product $C(X)\rtimes_\alpha G$, we write
\begin{equation} a = \sum_{g\in G} a_g u_g \end{equation}
to mean that $E(au_g^*)=a_g$ (an element of $C(X)$) for all $g \in G$, where $E:C(X)\rtimes_\alpha G \to C(X)$ is the canonical conditional expectation.
We do \emph{not} mean that the sum converges in any sense.

\begin{cor}
\label{cor:SeminormalizerCoeffs}
Let $G$ be a countable discrete group, let $X$ be a compact Hausdorff space, let $\alpha:G\curvearrowright X$ be a free action, and let
\begin{equation} a = \sum_{g\in G} a_g u_g \in C(X)\rtimes_\alpha G. \end{equation}
Then $a \in \mathcal{RN}_{C(X)\rtimes_\alpha G}(C(X))$ if and only if the collection 
\begin{equation} \{\suppo(a_g):g\in G\} \end{equation}
 is pairwise disjoint.
\end{cor}

\begin{proof}
The crossed product $C(X)\rtimes_\alpha G$ is the groupoid $C^*$-algebra of the transformation groupoid $\mathcal G=G \times X$ (see \cite[Example 2.1.15]{Sims:Notes} for example), and upon making this identification, one can easily compute
\[ \suppo(a) = \bigcup_{g\in G} \{g\} \times g^{-1}.\suppo(a_g). \]
Moreover, since $r(g,g^{-1}x) = x$ for $(g,x) \in \mathcal G$, this is an $r$-section if and only if $\{\suppo(a_g):g\in G\}$ is pairwise disjoint.
\end{proof}

We also use the above to give an interpretation of the conditions in Lemma \ref{lem:MatrixSeminormalizer} in the case of a free group action.

\begin{cor}
\label{cor:MatrixSeminormalizerCoeffs}
Let $G$ be a countable discrete group, let $X$ be a compact Hausdorff space, let $\alpha:G\curvearrowright X$ be a free action, and let
\begin{equation} x = \left(\begin{array}{ccc} x_{11} & \cdots & x_{1n} \\ \vdots && \vdots \\ x_{n1} & \cdots & x_{nn} \end{array}\right) \in M_n\otimes (C(X)\rtimes_\alpha G), \end{equation}
where for each $i,j=1,\dots,n$,
\begin{equation} x_{ij}=\sum_{g\in G} x_{i,j,g} u_g. \end{equation}
Then $x \in \mathcal{RN}_{M_n\otimes (C(X)\rtimes_\alpha G)}(D_n\otimes C(X))$ if and only if, for every $i=1,\dots,n$, the collection
\begin{equation} \{\suppo(x_{i,j,g}): j=1,\dots,n,\ g\in G\} \end{equation}
is pairwise disjoint.
\end{cor}

\begin{proof}
Consider the action $(\pi\times\alpha):\mathbb Z/n \times G \curvearrowright \{1,\dots,n\} \times X$ where $\pi$ is the canonical action of the cyclic group $\mathbb Z/n$ on $\{1,\dots,n\}$; this product action is free since both $\pi$ and $\alpha$ are.
The sub-$C^*$-algebra $( D_n\otimes C(X) \subseteq M_n\otimes (C(X)\rtimes_\alpha G) )$ identifies canonically with $(   C(\{1,\dots,n\}\times X) \subseteq C(\{1,\dots,n\}\times X) \rtimes_{\pi\times\alpha} (\mathbb Z/n \times G))$, and this identification maps $x$ to
\begin{equation} y:=\sum_{g\in G}\sum_{i,j=1}^n (\chi_{\{i\}} \otimes x_{i,j,g})u_{(i-j,g)} \in C(\{1,\dots,n\}\times X) \rtimes_{\pi\times\alpha} (\mathbb Z/n \times G). \end{equation}

Thus, $x \in \mathcal{RN}_{M_n\otimes (C(X)\rtimes_\alpha G)}(D_n\otimes C(X))$ if and only if 
\begin{equation} y \in\mathcal{RN}_{C(\{1,\dots,n\}\times X) \rtimes_{\pi\times\alpha} (\mathbb Z/n \times G)}(C(\{1,\dots,n\}\times X)). \end{equation}
By Corollary \ref{cor:SeminormalizerCoeffs}, this is equivalent to the collection
\begin{equation} \{\suppo(\chi_{\{i\}}\otimes x_{i,j,g}): i,j=1,\dots,n,\,g\in G\} \end{equation}
of subsets of $\{1,\dots,n\}\times G$ being pairwise disjoint.
Since $\suppo(\chi_{\{i\}}\otimes x_{i,j,g}) = \{i\}\times \suppo(x_{i,j,g})$, this is the same as requiring that
\begin{equation} \{\suppo(x_{i,j,g}): j=1,\dots,n,\ g\in G\} \end{equation}
be pairwise disjoint, for each $i$.
\end{proof}

Here is our algebraic characterization of the preorder $\preccurlyeq$ from Definition \ref{defn:DynamicalCuntz}. Note that although $\preccurlyeq$ is defined for diagonal matrices in $C(X)$ of different sizes, we may always pad one of them with zeroes to arrange that they have the same size.

\begin{prop}
\label{prop:DynamicalCuntzCharacterization}
Let $G$ be a countable discrete group, let $X$ be a compact Hausdorff space, and let $\alpha:G\curvearrowright X$ be a free action.
Let $a,b \in (D_n\otimes C(X))_+$.
The following are equivalent:
\begin{enumerate}
\item $a \preccurlyeq b$;
\item there exists a sequence $(t_k)_{k=1}^\infty$ in $\mathcal{RN}_{M_n\otimes (C(X)\rtimes_\alpha G)}(D_n\otimes C(X))$ such that
\begin{equation} \lim_{k\to\infty} \|t_k^*bt_k - a\|=0; \end{equation}
\item for every $\e>0$ there exists $\dl>0$ and $t \in \mathcal{RN}_{M_n\otimes (C(X)\rtimes_\alpha G)}(D_n\otimes C(X))$ such that
\begin{equation} t^*(b-\dl)_+t = (a-\e)_+. \end{equation}
\end{enumerate}
\end{prop}

\begin{proof}
Let us write $a=\diag(a_1,\dots,a_n)$ and $b=\diag(b_1,\dots,b_n)$.

(i) $\Rightarrow$ (iii):
This is a variant on the proofs of \cite[Lemma 12.3]{Kerr:JEMS} and \cite[Proposition 2.3]{Ma:ETDS}.
In both of those proofs, it is shown (roughly) that $a \preccurlyeq b$ implies that $a$ is Cuntz subequivalent to $b$ in $C(X)\rtimes_\alpha G$ (note that, in \cite[Lemma 12.3]{Kerr:JEMS}, the hypothesis is formally stronger than just $a \preccurlyeq b$); the main novelty here is to verify that the Cuntz subequivalence can be witnessed using $r$-normalizers.

Let $\e>0$ be given.
Set $F_i:=\overline{\suppo((a_i-\e)_+)}$ for $i=1,\dots,n$, so that $F_i$ is a compact set contained in $\suppo(a_i)$.
Then apply Definition \ref{defn:DynamicalCuntz} to obtain open sets $U_{i,j}\subseteq X$, group elements $s_{i,j} \in G$, and indices $k_{i,j}\in\{1,\dots,m\}$ for $i=1,\dots,n$ and $j=1,\dots,J_i$ such that
\[
F_i \subseteq U_{i,1}\cup \cdots \cup U_{i,J_i}
\]
for $i=1, \dots, n$, and
\begin{equation}
\label{eq:DynamicalCuntzCharacterization1}
\coprod_{i=1}^n \coprod_{j=1}^{J_i} s_{i,j}U_{i,j} \times \{k_{i,j}\} \subseteq \coprod_{l=1}^m \suppo(b_l) \times \{l\}.
\end{equation}
Next find open sets $V_{i,j}$ such that $\overline{V_{i,j}}\subseteq U_{i,j}$ and
\begin{equation}
\label{eq:DynamicalCuntzCharacterization2}
 F_i \subseteq V_{i,1}\cup \cdots \cup V_{i,J_i}, \quad i=1,\dots,n. \end{equation}
It follows that 
\begin{equation} \coprod_{i=1}^n \coprod_{j=1}^{J_i} \overline{s_{i,j}V_{i,j}} \times \{k_{i,j}\} \subseteq \coprod_{l=1}^m \suppo(b_l) \times \{l\}, \end{equation}
so by compactness of the left-hand side, there exists $\dl>0$ such that
\begin{equation} \coprod_{i=1}^n \coprod_{j=1}^{J_i} \overline{s_{i,j}V_{i,j}} \times \{k_{i,j}\} \subseteq \coprod_{l=1}^m \suppo((b_l-2\dl)_+) \times \{l\}. \end{equation}
By \eqref{eq:DynamicalCuntzCharacterization2} and the definition of $F_i$, we may choose a continuous function $h_{i,j} \in C_0(V_{i,j})_+$ for each $i=1,\dots,n$ and $j=1,\dots,J_i$ such that
\begin{equation}
\label{eq:hDef}
 \sum_{j=1}^{J_i} h_{i,j}^2 = (a_i-\e)_+, \quad i=1,\dots,n. \end{equation}
Using functional calculus, let $\hat{b}_l \in C^*(b_l)$ be a function such that
\begin{equation}
\label{eq:hatbDef}
 \hat{b}_l(x)^2(b_l-\dl)_+(x) = 1,\quad x \in \suppo((b_l-2\dl)_+). \end{equation}
Now define
\begin{equation} t = \left(\begin{array}{ccc} t_{11} & \cdots & t_{1n} \\ \vdots && \vdots \\ t_{n1} & \cdots & t_{nn} \end{array}\right) \in M_n \otimes (C(X)\rtimes_\alpha G) \end{equation}
by
\begin{equation} t_{li} := \sum_{j: k_{i,j}=l} \hat{b}_lu_{s_{i,j}}h_{i,j} = \sum_{j:k_{i,j}=l} \hat{b}_l(h_{i,j}\circ\alpha_{s_{i,j}}^{-1})u_{s_{i,j}},\quad i,l=1,\dots,n. \end{equation}
By \eqref{eq:DynamicalCuntzCharacterization1} and since $h_{i,j} \in C_0(V_{i,j})$, for each $l$ the collection
\begin{equation}
\label{eq:DynamicalCuntzCharacterization4}
 \{\suppo(h_{i,j}\circ\alpha_{s_{i,j}}^{-1}): k_{i,j}=l\} = \{\alpha_{s_{i,j}}(\suppo(h_{i,j})): k(i,j)=l\} \end{equation}
is pairwise disjoint.
Therefore by Corollary \ref{cor:MatrixSeminormalizerCoeffs}, $t \in \mathcal{RN}_{M_n\otimes (C(X)\rtimes_\alpha G)}(D_n\otimes C(X))$.

In particular, $t^*(b-\dl)_+t$ is a diagonal matrix.
Moreover, for $i=1,\dots,n$, we compute the $(i,i)$-entry of $t^*(b-\dl)_+t$ to be
\begin{eqnarray}
\notag
\sum_{k=1}^n t_{ki}^*(b_k-\dl)_+t_{ki}
&=& \sum_{j,j': k_{i,j}=k_{i,j'}} h_{i,j}u_{s_{i,j}}^*\hat{b}_{k_{i,j}}^2(b_{k_{i,j}}-\dl)_+u_{s_{i,j'}}h_{i,j'} \\
&\stackrel{\eqref{eq:hatbDef}}=& \sum_{j,j': k_{i,j}=k_{i,j'}} u_{s_{i,j}}^*(h_{i,j}\circ\alpha_{s_{i,j}}^{-1})(h_{i',j'}\circ\alpha_{s_{i,j'}}^{-1})u_{s_{i,j'}}.
\end{eqnarray}
By pairwise disjointness of the collection \eqref{eq:DynamicalCuntzCharacterization4}, we have that $(h_{i,j}\circ\alpha_{s_{i,j}}^{-1})(h_{i,j'}\circ\alpha_{s_{i,j'}})=0$ whenever $k_{i,j}=k_{i,j'}$ and $j\neq j'$; thus the above simplifies to
\begin{equation}
\sum_{j=1}^{J_i} u_{s_{i,j}}^*(h_{i,j}\circ\alpha_{s_{i,j}}^{-1})^2u_{s_{i,j}} 
\stackrel{\eqref{eq:hDef}}= \sum_{j=1}^{J_i} h_{i,j}^2 = (a_i-\e)_+,
\end{equation}
as required.

(iii) $\Rightarrow$ (i):
As in Definition \ref{defn:DynamicalCuntz}, let $F_i$ be a compact subset of $\suppo(a_i)$ for $i=1,\dots,n$.
By compactness, there exists $\e>0$ such that $F_i \subseteq \suppo((a_i-\e)_+)$.
By (iii), let $t \in \mathcal{RN}_{M_n\otimes (C(X)\rtimes_\alpha G)}(D_n\otimes C(X))$ and $\dl>0$ be such that $t^*(b-\dl)_+t=(a-\e)_+$.
Write
\begin{equation} t = \left(\begin{array}{ccc} t_{11} & \cdots & t_{1n} \\ \vdots && \vdots \\ t_{n1} & \cdots & t_{nn} \end{array}\right) \end{equation}
and for each $i,j$, write
\begin{equation} t_{ij}=\sum_{g\in G} t_{i,j,g} u_g \end{equation}
where $t_{i,j,g} \in C(X)$ for each $i,j,g$.
By Corollary \ref{cor:MatrixSeminormalizerCoeffs}, for each $i$
\begin{equation}
\label{eq:DynamicalCuntzCharacterization3}
 \{\suppo(t_{i,j,g}): j=1,\dots,n,\ g\in G\}\text{ is pairwise disjoint}. \end{equation}
We now compute that $(a_i-\e)_+$, which is the $(i,i)$-entry of $t^*(b-\dl)_+t$, is equal to
\begin{eqnarray}
\notag
\sum_{k=1}^n t_{ki}^*(b_k-\dl)_+t_{ki}
&=& \sum_{k=1}^n \sum_{g,h\in G} u_g^* t_{k,i,g}^*(b_k-\dl)_+t_{k,i,h} u_h \\
\notag
&\stackrel{\eqref{eq:DynamicalCuntzCharacterization3}}=& \sum_{k=1}^n \sum_{s \in G} u_s^* |t_{k,i,s}|^2b_k u_s \\
&=& \sum_{k=1}^n \sum_{s\in G} (|t_{k,i,s}|^2(b_k-\dl)_+)\circ\alpha_s.
\end{eqnarray}
Since $F_i \subseteq \suppo((a_i-\e)_+)$, it follows that
\begin{equation} F_i \subseteq \bigcup_{k=1}^n \bigcup_{s\in G} \suppo((|t_{k,i,s}|^2b_k)\circ\alpha_s). \end{equation}
By compactness, we may choose $k_{i,1},\dots,k_{i,J_i} \in \{1,\dots,n\}$ and $s_{i,1},\dots,s_{i,J_i}\in G$ such that, upon setting
\begin{equation} U_{i,j} := \suppo((|t_{k_{i,j},i,s_{i,j}}|^2b_{k_{i,j}})\circ\alpha_{s_{i,j}}) = \alpha_{s_{i,j}}^{-1}(\suppo(t_{k_{i,j},i,s_{i,j}})\cap \suppo(b_{k_{i,j}})), \end{equation}
we have
\begin{equation} F_i \subseteq \bigcup_{j=1}^{J_i} U_{i,j}. \end{equation}

Also, the collection of sets of the form
\begin{equation} \{k_{i,j}\}\times \alpha_{s_{i,j}}(U_{i,j}) = \{k_{i,j}\} \times (\suppo(t_{k_{i,j},i,s_{i,j}}) \cap \suppo(b_{k_{i,j}})) \end{equation}
(where $i$ ranges over $\{1,\dots,n\}$ and $j$ ranges over $\{1,\dots,J_i\}$)
is contained in the collection of sets of the form
\begin{equation} \{k\} \times (\suppo(t_{k,i,s})\cap \suppo(b_k)). \end{equation}
Each of these is evidently contained in $\coprod_k \{k\} \times \suppo(b_k)$, and by \eqref{eq:DynamicalCuntzCharacterization3}, they are pairwise disjoint.
Therefore, we have
\begin{equation} \coprod_{i,j} \{k_{i,j}\}\times \alpha_{s_{i,j}}(U_{i,j}) \subseteq \coprod_k \{k\} \times \suppo(b_k), \end{equation}
as required.

(iii) $\Rightarrow$ (ii) is immediate.

(ii) $\Rightarrow$ (iii):
Assume (ii) holds and let $\e>0$ be given.
Then for some $k$ we have $\|t_k^*bt_k-a\|<\e/2$, and thus there exists $\dl>0$ such that $\|t_k^*(b-\dl)_+t_k-a\|<\e$.
Since $t_k \in \mathcal{RN}_{M_n\otimes (C(X)\rtimes_\alpha G)}(D_n\otimes C(X))$, it follows that $t_k^*(b-\dl)_+t_k \in D_n\otimes C(X)$, so applying \cite[Proposition 2.2]{Rordam:JFA92} to this algebra (and using that it is commutative), we see that there exists $s \in D_n\otimes C(X)$ such that
\begin{equation} s^*t_k^*(b-\dl)_+t_ks = (a-\e)_+. \end{equation}
Thus (iii) holds with $t:=t_ks$.
\end{proof}

\begin{remark}
\label{rmk:Kerr12.3}
As noted in the above proof, the argument for (i) $\Rightarrow$ (iii) is a variant on the proof of \cite[Lemma 12.3]{Kerr:JEMS}.
We note for use later that, in fact, the element $v$ constructed in the proof of \cite[Lemma 12.3]{Kerr:JEMS} is an $r$-normalizer, for example by writing it as
\begin{equation} v=\sum_{i=1}^n ((fh_i)^{1/2}\circ\alpha_{s_i}^{-1}) u_{s_i} \end{equation}
(where we use $\alpha:G \curvearrowright X$ to denote the action) and then using Corollary \ref{cor:SeminormalizerCoeffs}.
\end{remark}

\section{Almost finiteness, dynamical comparison, and tracial $\mathcal Z$-stability}
\label{sec:Z}

In this section we define tracial $\mathcal{Z}$-stability for sub-$C^*$-algebras and prove Theorem \ref{thm:A} (as Corollary \ref{cor:AlmostFiniteZstable}). We first recall the definition of dynamical comparison.

\begin{defn}[{\cite[Definition 3.2]{Kerr:JEMS}}] \label{defn:DynamicalComparison}
Let $G$ be a countable discrete group, let $X$ be a compact metrizable space, and let $\alpha:G\curvearrowright X$ be an action by homeomorphisms. We say $\alpha$ has \emph{dynamical comparison} if $O \prec V$ for all open sets $O, V\subseteq X$ satisfying $\mu(O) < \mu(V)$ for every $G$-invariant Borel probability measure $\mu$.
\end{defn}

Given $\tau \in T(C(X)\rtimes_\alpha G)$, and $a \in C(X)_+$, define
\begin{equation} d_\tau(a):=\lim_{n\to\infty} \tau(a^{1/n}), \end{equation}
i.e., the value of the measure associated to $\tau$ evaluated on $\suppo(a)$.
When the action $\alpha$ is \emph{free}, the $G$-invariant Borel probability measures on $X$ correspond exactly to tracial states on $C(X)\rtimes_\alpha G$ (see, for example, \cite[Theorem 11.1.22]{GKPT:book}). Therefore in this case dynamical comparison can be reformulated as follows: the action $\alpha$ has dynamical comparison if and only if for any $a,b \in C(X)_+$, if $d_\tau(a) < d_\tau(b)$ for all $\tau \in T(C(X)\rtimes_\alpha G)$ then $a \preccurlyeq b$ (as in Definition \ref{defn:DynamicalCuntz}).

In $C^*$-algebra theory, $\mathcal Z$-stability is the property of tensorially absorbing a certain canonical $C^*$-algebra, called the Jiang--Su algebra $\mathcal Z$.
This $C^*$-algebra was defined in \cite{JiangSu}, but in practice it is often a McDuff-type characterization of $\mathcal Z$-stability that is used (see \cite[Proposition 2.14]{Winter12}, a combination of results by R\o rdam--Winter \cite{RordamWinter}, Kirchberg \cite{Kir06}, and Toms--Winter \cite{TomsWinter}).
Building on ideas of Matui and Sato (\cite{MS:Acta}), Hirshberg and Orovitz defined ``tracial $\mathcal Z$-stability'', an a priori weakening of this McDuff-type condition, and proved that it is equivalent to $\mathcal Z$-stability for simple separable unital nuclear $C^*$-algebras (see \cite[Definition 2.1, Proposition 2.2, and Theorem 4.1]{HirshbergOrovitz}) (nuclearity is the key hypothesis that enables this equivalence).
The following is a version of tracial $\mathcal Z$-stability for a sub-$C^*$-algebra, where a key role is played by (one-sided) normalizers of the smaller algebra; tracial $\mathcal Z$-stability for the algebra $A$ is precisely the case $D=A$.

\begin{defn}
\label{defn:Zstable}
Let $(D \subseteq A)$ be a sub-$C^*$-algebra with $A$ unital, and such that $\unit_A \in D$.
For $a,b \in D_+$, write $a \precsim_{(D \subseteq A)} b$ if there is a sequence $(t_k)_{k=1}^\infty$ in $\mathcal{RN}_A(D)$ such that $\lim_{k\to\infty} \|t_k^*bt_k - a\|=0$.

We say that $(D \subseteq A)$ is \emph{tracially $\mathcal Z$-stable} if for every $n \in \N$, every tolerance $\e>0$, every finite set $\mathcal F \subset A$, and every $h \in D_+ \setminus \{0\}$, there exists a c.p.c.\ order zero map $\phi:M_n \to A$ such that:
\begin{enumerate}
\item $\phi(\mathcal N_{M_n}(D_n)) \subseteq \mathcal N_A(D)$,
\item $\unit_A - \phi(\unit_n) \precsim_{(D \subseteq A)} h$,
\item $\|[a,\phi(x)]\| < \e$ for all $a \in \mathcal F$ and every contraction $x \in M_n$.
\end{enumerate}
(Note in (ii) that, by (i) and Lemma \ref{lem:NormalizersPreserving} $\phi(\unit_n) \in D$, so the statement makes sense.)
\end{defn}

We will make use of the generalized type semigroup defined in \cite{Ma:ETDS}. The main result of \cite{Ma:ETDS} shows that dynamical comparison is equivalent to almost unperforation of this semigroup.

\begin{defn}[{\cite[the paragraphs after Lemma 2.2]{Ma:ETDS}}]
 Let $G$ be a countable 	discrete group, let $X$ be a compact Hausdorff space, and let $\alpha:G\curvearrowright X$ be a free action by homeomorphisms.
	Define $a \approx b$ to mean that both $a \preccurlyeq b$ and $b \preccurlyeq a$ (as in Definition \ref{defn:DynamicalCuntz}), and set
	\begin{equation} W(X,G):=\bigcup_{n=1}^\infty (D_n\otimes C(X))_+/\approx. \end{equation}
	For $a \in (D_n\otimes C(X))_+$, we use $[a]$ to denote its equivalence class in $W(X,G)$.
	The preorder $\preccurlyeq$ induces an order $\leq$ on $W(X,G)$, and there is a well-defined addition operation on $W(X,G)$ given by
	\begin{equation} [a] + [b] = [a \oplus b]. \end{equation}
	$W(X, G)$ is then a partially ordered abelian semigroup.
\end{defn}

\begin{thm} 
	\label{thm:ZstableComparison}
	Let $G$ be a countable discrete infinite amenable group, let $X$ be a compact metrizable space, and let $\alpha:G \curvearrowright X$ be a free minimal action.
	If $(C(X) \subseteq C(X)\rtimes_\alpha G)$ is tracially $\mathcal Z$-stable then $\alpha$ has dynamical comparison.
\end{thm}

\begin{proof}
	Let $f,g \in C(X)_+$, and assume that $d_\tau(f)<d_\tau(g)$ for all $\tau \in T(C(X)\rtimes_\alpha G)$.
	We need to show that $f \preccurlyeq g$ (as in Definition \ref{defn:DynamicalCuntz}).
	
	Since the action is free and minimal and $G$ is infinite, $X$ has no isolated points.
	Choose any point $x_0 \in X$ such that $g(x_0)\neq 0$, let $h \in C(X)_+$ be a positive contraction which vanishes at $x_0$ and is nonzero everywhere else, and consider $g':=hg$.
	Since the action is minimal, $\mu(\{x_0\})=0$ for every $G$-invariant measure on $X$.
	Since the $G$-invariant probability measures on $X$ correspond to traces on $C(X)\rtimes_\alpha G$, it follows that $d_\tau(g)=d_\tau(g')$ for all $\tau \in T(C(X)\rtimes_\alpha G)$.
	Also, $0$ is not an isolated point of the spectrum of $g'$.
	Thus by replacing $g$ with $g'$, we may assume that $0$ is not an isolated point in the spectrum of $g$. 
	
	Now by \cite[Theorem 3.9]{Ma:ETDS}, and again since the $G$-invariant probability measures on $X$ correspond to traces on $C(X)\rtimes_\alpha G$, the condition $d_\tau(f)<d_\tau(g)$ for all $\tau \in T(C(X)\rtimes_\alpha G)$ is equivalent to $(n+1)[f]\leq n[g]$ in $W(X,G)$, for some $n \in \mathbb N$.
	
	Since $n[f] \leq (n+1)[f]\leq n[g]$, this means that $\unit_{M_n}\otimes f \preccurlyeq \unit_{M_n} \otimes g$.
	Let $\e>0$; we will show that there exists $t \in \mathcal{RN}_{C(X)\rtimes_\alpha G}(C(X))$ such that $\|t^*gt-f\|<\e$, which suffices to show $f \preccurlyeq g$ by Proposition \ref{prop:DynamicalCuntzCharacterization}.
	By that same proposition, we have that there exists $\dl>0$ and
	\begin{equation} v=\left(\begin{array}{ccc} v_{11} &\cdots & v_{1n} \\ \vdots&&\vdots \\ v_{n1} &\cdots & v_{nn} \end{array}\right) \in \mathcal{RN}_{M_n\otimes C(X)\rtimes_\alpha G}(D_n \otimes C(X)) \end{equation}
	such that $v^*(\unit_{M_n}\otimes (g-\dl)_+)v=\unit_{M_n}\otimes (f- \frac{\e}{2})_+$.
	By Lemma \ref{lem:MatrixSeminormalizer}, we have
	\begin{align}
	\label{eq:Comparison1}
	&v_{ij} \in \mathcal{RN}_{C(X)\rtimes_{ \alpha } G}(C(X)),\quad &&i,j=1,\dots,n, \text{ and} \\
	\label{eq:Comparison2}
	&v_{ki}^*av_{kj}=0,\quad && i\neq j,\ a\in C(X).
	\end{align}
	By looking at the entries of $v^*(\unit_{M_n}\otimes (g-\dl)_+)v=\unit_{M_n}\otimes (f-\frac{\e}{2})_+$, we obtain for all $j$
	\begin{equation}
	\label{eq:Comparison3}
	\sum_{i=1}^n v_{ij}^*(g-\dl)_+v_{ij} = (f-\frac{\e}{2})_+.
	\end{equation}
	
	Since $0$ is not an isolated point of the spectrum of $g$, we may use functional calculus to find a nonzero element $d \in C^*(g)_+ \subseteq C(X)_+$, along with orthogonal elements $\hat{d},\hat{g} \in C^*(g)_+\subseteq C(X)_+$ such that
	\begin{equation}
	\label{eq:Comparison4}
	g\hat{d}^2 = d \quad \text{and} \quad g\hat{g}^2 = (g-\dl)_+. \end{equation}
	
	Set
	\begin{equation} \eta:=\frac\e{8n^2+3} \end{equation}
	and using tracial $\mathcal Z$-stability of $(C(X) \subseteq C(X)\rtimes_\alpha G)$, let $\phi:M_n\to C(X)\rtimes_\alpha G$ be an order zero map such that:
	\begin{enumerate}
		\item $\phi(\mathcal N_{M_n}(D_n)) \subseteq \mathcal N_{C(X)\rtimes_\alpha G}(C(X))$,
		\item $\unit_A - \phi(\unit_n) \preccurlyeq d$ (note that $\precsim_{(C(X) \subseteq C(X)\rtimes_\alpha G)}$ in Definition \ref{defn:Zstable} is the same as $\preccurlyeq$ from Definition \ref{defn:DynamicalCuntz}, by Proposition \ref{prop:DynamicalCuntzCharacterization}), and 
		\item $\|[a,\phi(x)]\| < \eta$ for all $a \in \{v_{ij}:i,j=1,\dots,n\} \cup \{(g-\dl)_+\}$ and every contraction $x \in M_n$.
	\end{enumerate}
	Define
	\begin{equation} r:=\sum_{i,j=1}^n \phi(e_{ii})v_{ij}\phi(e_{ij}). \end{equation}
	Note that by (i) and Lemma \ref{lem:NormalizersPreserving}, $\phi(e_{ii})\in C(X)$.
	Using this, for $a \in C(X)$, we have
	\begin{eqnarray}
	\notag
	r^*ar &=& \sum_{i,j,k,l=1}^n \phi(e_{ji})v_{ij}^*\phi(e_{ii})a\phi(e_{kk})v_{kl}\phi(e_{kl}) \\
	\notag
	&=& \sum_{i,j,k,l=1}^n \phi(e_{ji})v_{ij}^*\phi(e_{ii})\phi(e_{kk})av_{kl}\phi(e_{kl}) \\
	\notag
	&=& \sum_{i,j,l=1}^n \phi(e_{ji})v_{ij}^*\phi(e_{ii})^2av_{il}\phi(e_{il}) \\
	&\stackrel{\eqref{eq:Comparison2}}=& \sum_{i,j=1}^n \phi(e_{ji})v_{ij}^*\phi(e_{ii})^2av_{ij}\phi(e_{ij}),
	\end{eqnarray}
	and by \eqref{eq:Comparison1}, and (i), this is in $C(X)$.
	This shows that $r \in \mathcal{RN}_{C(X)\rtimes_\alpha G}(C(X))$.
	
	Next, in the case that $a=(g-\dl)_+$, we get
	\begin{eqnarray}
	\notag
	r^*(g-\dl)_+r
	\notag
	&=& \sum_{i,j=1}^n \phi(e_{ji})v_{ij}^*\phi(e_{ii})(g-\dl)_+v_{ij}\phi(e_{ij}) \\
	\notag
	&\capprox{4n^2\eta}& \sum_{i,j=1}^n \phi(e_{ji})\phi(e_{ii})^2\phi(e_{ij})v_{ij}^*(g-\dl)_+v_{ij} \\
	\notag
	&=& \sum_{i,j=1}^n \phi(e_{jj})^4v_{ij}^*(g-\dl)_+v_{ij} \\
	\notag
	&\stackrel{\eqref{eq:Comparison3}}=& \sum_{j=1}^n \phi(e_{jj})^4 (f-\frac{\e}{2})_+ \\
	&\capprox{\e/2}& \phi(\unit_{M_n})^4 f.
	\end{eqnarray}
	Next, we note that 
$\unit_{C(X)\rtimes_\alpha G}-\phi(\unit_{M_n})^4$ is Cuntz subequivalent to $\unit_{C(X)\rtimes_\alpha G}-\phi(\unit_{M_n})$ in $C^*(\unit_{C(X)\rtimes_\alpha G},\phi(\unit_{M_n}))$ (in fact, they are Cuntz equivalent), and so combining this with (ii), we obtain some $s \in \mathcal{RN}_{C(X)\rtimes_\alpha G}(C(X))$ such that
	\begin{equation} s^*ds \approx_\eta (\unit_{C(X)\rtimes_\alpha G}-\phi(\unit_{M_n})^4)f. \end{equation}
	Define $t:=\hat{g}r+\hat{d}s$ (using $\hat{g},\hat{d}$ defined just above \eqref{eq:Comparison4}).
	Since $\hat{g},\hat{d}$ are orthogonal and in $C(X)$, it follows that $(\hat{g}r)^*a(\hat{d}s)=0$ for all $a \in C(X)$, and thus $t\in \mathcal{RN}_{C(X)\rtimes_\alpha G}(C(X))$ by Lemma \ref{lem:SumOfNormalizers}.
	Moreover,
	\begin{eqnarray}
	\notag
	t^*gt &=& r^*\hat{g}g\hat{g}r+s^*\hat{d}g\hat{d}s \\
	\notag
	&\stackrel{\eqref{eq:Comparison4}}=& r^*(g-\dl)_+r + s^*ds \\
	&\capprox{4n^2\eta+\frac{\e}{2}+\eta}& \phi(\unit_{M_n})^4f+(1-\phi(\unit_{M_n})^4)f = f.
	\end{eqnarray}
	Since $4n^2\eta+\eta<\frac\e2$, we are done.
\end{proof}

We now establish Theorem \ref{thm:A}. The following definition of a ``castle'' is borrowed from David Kerr, except that (for later use) we allow the castle to possibly have infinitely many towers.

\begin{defn}[{cf.\ \cite[Definitions 4.1 and 5.7]{Kerr:JEMS}}]
Let $G$ be a countable discrete group, let $X$ be a compact Hausdorff space, and let $\alpha:G\curvearrowright X$ be a free action.
A \emph{castle} is a collection $\{(V_i,S_i)\}_{i\in I}$ where each $V_i$ is a subset of $X$ and each $S_i$ is a finite subset of $G$, such that the collection
\begin{equation} \{sV_i: s \in S_i,\, i \in I\} \end{equation}
is pairwise disjoint.
Each $(V_i,S_i)$ is called a \emph{tower}, the sets $S_i$ are called \emph{shapes}, and the sets $sV_i$ (where $s \in S_i$) are called \emph{levels} of the castle.
\end{defn}

We recall the definition of almost finiteness for a group acting by homeomorphisms.
Kerr's definition partially generalizes an earlier concept for locally compact \'etale groupoids with compact totally disconnected unit spaces, which Matui defined and used to prove strong results about the associated topological full group (see \cite{Matui:PLMS}, particularly Definition 6.2).

\begin{defn}[{\cite[Definition 8.2]{Kerr:JEMS}}]
Let $G$ be a countable discrete group, let $X$ be a compact metrizable space, and let $\alpha:G\curvearrowright X$ be a free action.
The action is \emph{almost finite} if for every finite subset $K \subset G$, and every $\dl>0$, there exists:
\begin{enumerate}
\item a castle $\{(V_i,S_i)\}_{i\in I}$ such that $I$ is finite, each level is open with diameter at most $\dl$, and each shape is $(K,\dl)$-invariant (i.e., $|gS_i \triangle S_i|/|S_i| < \dl$ for all $g \in K$ and all $i\in I$), and
\item a set $S_i' \subseteq S_i$ for each $i\in I$ such that $|S_i'|<\dl |S_i|$ and
\begin{equation} X \setminus \coprod_{i\in I} S_iV_i \prec \coprod_{i\in I} S_i'V_i, \end{equation}
using $\prec$ from Definition \ref{defn:DynamicalCuntz}.
\end{enumerate}
\end{defn}

The following corollary shows how tracial $\mathcal{Z}$-stability (for sub-$C^*$-algebras) fits into different regularity-type dynamical properties.

\begin{cor} (Theorem \ref{thm:A}) \label{cor:AlmostFiniteZstable}
Let $G$ be a countable discrete amenable group, let $X$ be a compact metrizable space, and let $\alpha:G \curvearrowright X$ be a free minimal action. Consider the following conditions:
\begin{enumerate}
	\item $\alpha$ is almost finite;
	\item $(C(X) \subseteq C(X)\rtimes_\alpha G )$ is tracially $\mathcal Z$-stable;
	\item $\alpha$ has dynamical comparison.
\end{enumerate}
Then $(i)\Rightarrow (ii) \Rightarrow (iii)$, and if $\alpha$ has the small boundary property then all three conditions are equivalent.
\end{cor}

As mentioned in the introduction, our definition of tracial $\mathcal Z$-stability for a sub-$C^*$-algebra is largely inspired by Kerr's proof that almost finiteness implies $\mathcal Z$-stability of $C(X)\rtimes_\alpha G$ (\cite[Theorem 12.4]{Kerr:JEMS}), and indeed his proof shows the implication (i) $\Rightarrow$ (ii) (we flesh out the details below).
For the rest, one need only combine Theorem \ref{thm:ZstableComparison} with Kerr and Szab\'o's proof that dynamical comparison combined with the small boundary property implies almost finiteness (\cite[Theorem 6.1]{KerrSzabo}).

To verify (i) $\Rightarrow$ (ii), we take a closer look at the normalizer-preserving condition in the definition of $\mathcal Z$-stability for a sub-$C^*$-algebra.
We shall consider a general construction of a c.p.c.\ map into a crossed product, given the following data.
Let $G$ be a countable discrete group, let $X$ be a compact metrizable space, and let $\alpha:G \curvearrowright X$ be an action.
Let $T$ be a countable set, and for each $t\in T$, let $f_t \in C(X)_+$ and let $S_t=\{s_{t,1},\dots,s_{t,n}\}$ be a subset of $G$ of size $n$.
Suppose that:
\begin{enumerate}
\item $\lim_{t\to\infty} \|f_t\| = 0$, and
\item $((\suppo(f_t),S_t)_{t=1}^\infty$ is a castle.
\end{enumerate}
Also for each $t\in T$ and $i=1,\dots,n$, let $\theta_{t,i}:\suppo(f_t) \to \mathbb T$ be a continuous function.
Define $\phi:M_n \to C(X)\rtimes_\alpha G$ by
\begin{equation} \phi(e_{ij}):=\sum_{t\in T} u_{s_{t,i}}\theta_{t,i}\bar\theta_{t,j}f_tu_{s_{t,j}}^* \end{equation}
and extending linearly.
We call such a map a \emph{castle order zero map}. It is not hard to check the following facts.
\begin{enumerate}
\item The sum defining $\phi(e_{ij})$ converges in norm.
\item $\phi$ is c.p.c.\ order zero and normalizer-preserving, i.e., $\phi(\mathcal N_{M_n}(D_n)) \subseteq \mathcal N_{C(X)\rtimes_\alpha G}(C(X))$ (by Lemma \ref{lem:OrderZeroMatrixUnit} it is enough to check that $\phi(e_{ij})$ is a normalizer for all $i,j$).
\end{enumerate}

\begin{remark}
\label{remark:KerrOrderZeroMap} We check that the map $\varphi$ defined in Equation (18) in the proof of \cite[Theorem 12.4]{Kerr:JEMS} is a castle order zero map (again with a finite sum). To see this, using the notation of that proof, set
\begin{align}
I&:=\{(k,l,m,q,c,t): k=1,\dots,K,\,l=1,\dots,L,\\
\notag
&\qquad\qquad 
m=1,\dots,M,\,
c\in C_{k,l,m}^{(1)},\,q=1,\dots,Q,\,t\in B_{k,l,c,q}\},
\end{align}
and for $(k,l,m,q,c,t) \in I$, set
\begin{align}
\notag
S_{(k,l,m,q,c,t)} &:= \{t\Lambda_{k,1}(c),\dots,t\Lambda_{k,n}(c)\}, \\
f_{(k,l,m,q,c,t)} &:= \frac qQ h_k, \quad \text{and}\quad \theta_{t,i} \equiv 1.
\end{align}
We note that for each $k$, the sets $S_{(k,l,m,q,c,t)}$ are pairwise disjoint\footnote{Suppose that $x \in S_{(k,l,m,q,c,t)}\cap S_{(k,l',m',q',c',t')}$.
Then $x= t\Lambda_{k,i}(c)=t'\Lambda_{k,i'}(c')$ for some $i,i'\in\{1,\dots,n\}$.
We have $t \in B_{k,l,c,q} \subseteq T_{k,l,c}'\subseteq T_{k,l,c}$ and $t' \in B_{k',l',c',q'} \subseteq T_{k',l',c'}'\subseteq T_{k',l',c'}$, while $\Lambda_{k,i}(c)\in C_{k,l,m}^{(i)}\subseteq C_{k,l}$ and $\Lambda_{k',i'}(c') \in C_{k',l',m'}^{(i')}\subseteq C_{k',l'}$.
Since the collection of sets $T_{k,l,c}\gamma$ for $l=1,\dots,L$ and $\gamma\in C_{k,l}$ are disjoint, it follows that $l=l'$, $t=t'$, and $\Lambda_{k,i}(c)=\Lambda_{k,i'}(c')$.
Since $C_{k,l,m}^{(i)}$ for $i=1,\dots,n$ and $m=1,\dots,M$ are pairwise disjoint, it then follows that $i=i'$ and $m=m'$.
Injectivity of $\Lambda_{k,i}$ then implies that $c=c'$.
Finally, the sets $B_{k,l,c,q}$ for $q=1,\dots,Q$ are pairwise disjoint, so $q=q'$.}
and contained in $S_k$ (where $S_k$ is defined in \cite{Kerr:JEMS})\footnote{Continuing from the previous footnote, if $x= t\Lambda_{k,i}(c)$ for some $i\in\{1,\dots,n\}$ then $x \in T_{k,l,c}C_{k,l,m} \subseteq S_k$ (by the use of \cite[Theorem 12.2]{Kerr:JEMS} right after \cite[Eq.\ (14)]{Kerr:JEMS}).}
Since $h_k \in C_0(U_k)$ and the sets $sU_k$ for $k=1,\dots,K$ are pairwise disjoint, it follows that
\begin{equation} (\suppo(f_{(k,l,m,q,c,t)}),S_{(k,l,m,q,c,t)})_{(k,l,m,q,c,t)\in I} \end{equation}
is a castle, and so defines a castle order zero map.

To see that the map it defines is $\varphi$, using that $\Lambda_{k,i,j} = \Lambda_{k,i} \circ \Lambda_{k,j}^{-1}$ and that $\Lambda_{k,j}:C_{k,l,m}^{(1)} \to C_{k,l,m}^{(j)}$ is a bijection, we see that
\begin{align}
\notag
 h_{k,l,c,i,j} &= \sum_{q=1}^Q \sum_{t\in B_{k,l,c,q}} \frac qQ u_{t\Lambda_{k,i,j}(c)c^{-1}t^{-1}}(h_k\circ\alpha_{tc}) \\
&= \sum_{q=1}^Q \sum_{t\in B_{k,l,c,q}} u_{t\Lambda_{k,i}(c)} f_{(k,l,m,q,c,t)} u_{t\Lambda_{k,j}(c)}^*.
\end{align}
Thus
\begin{equation}
\varphi(e_{ij}) = \sum_{(k,l,m,q,c,t) \in I} u_{t\Lambda_{k,i}(c)} f_{(k,l,m,q,c,t)} u_{t\Lambda_{k,j}(c)}^*.
\end{equation}
\end{remark}

\begin{proof}[Proof of Corollary \ref{cor:AlmostFiniteZstable}]
As explained earlier, (ii) $\Rightarrow$ (iii) is Theorem \ref{thm:ZstableComparison}. When $\alpha$  has the small boundary property, (iii) $\Rightarrow$ (i) is \cite[Theorem 6.1]{KerrSzabo}.

(i) $\Rightarrow$ (ii): 
The proof of \cite[Theorem 12.4]{Kerr:JEMS} essentially shows this implication, although since tracial $\mathcal Z$-stability for a sub-$C^*$-algebra is not defined there, it is not explicitly stated in this way.
Let us explain carefully how to obtain (ii) from the proof of \cite[Theorem 12.4]{Kerr:JEMS}.

The proof begins with $a \in C(X)_+$ nonzero, a finite set $\Gamma$ of the unit ball of $C(X)$, a finite symmetric set $F$ of $G$, and a tolerance $\e>0$.
It produces a c.p.c.\ order zero map $\phi:M_n \to C(X)\rtimes_\alpha G$ such that:
\begin{itemize}
\item $\unit_{C(X)\rtimes_\alpha G} -\phi(\unit_n)$ is Cuntz subequivalent to $a$, and
\item $\|[x,\phi(b)]\|<\e$ for all $x \in \Gamma \cup \{u_g:g\in F\}$.
\end{itemize}
As argued in the proof of \cite[Theorem 12.4]{Kerr:JEMS}, since the crossed product is generated by $C(X)$ and the canonical unitaries, given any finite subset $F'$ of $C(X)\rtimes_\alpha G$, by choosing $\Gamma$, $F$, and $\e$ appropriately, the second condition will imply that $\|[x,\phi(b)]\|<\e$ for all $x \in F'$.
By Remark \ref{remark:KerrOrderZeroMap}, the order zero map $\phi$ is in fact a ``castle order zero map'' and therefore $\phi(\mathcal N_{M_n}(D_n)) \subseteq \mathcal N_{C(X)\rtimes_\alpha G}(C(X))$.
Finally, we note that the first of the above conditions is obtained in the proof of \cite[Theorem 12.4]{Kerr:JEMS} by invoking \cite[Lemma 12.3]{Kerr:JEMS}, and so by Remark \ref{rmk:Kerr12.3}, we get the stronger conclusion that $\unit_{C(X)} - \phi(\unit_n) \precsim_{(C(X)\subseteq C(X)\rtimes_\alpha G )} a$.
Hence, we get precisely our definition of $(C(X)\subseteq C(X)\rtimes_\alpha G )$ being tracially $\mathcal Z$-stable, as required.
\end{proof}

We close by noting that normalizer-preserving c.p.c.\ order zero maps from $M_n$ into the crossed product $C(X)\rtimes_\alpha G$ are closely related to castles. In fact, every such a map must be a castle order zero map.

\begin{prop}
	Let $G$ be a countable discrete amenable group, let $X$ be a compact metrizable space, and let $\alpha:G \curvearrowright X$ be a free action.
	If $\phi:M_n \to C(X)\rtimes_\alpha G$ is a c.p.c.\ order zero map such that $\phi(\mathcal N_{M_n}(D_n)) \subseteq \mathcal N_{C(X)\rtimes_\alpha G}(C(X))$, then $\phi$ is a castle order zero map.
\end{prop}

\begin{proof}
	Let
	\begin{equation} \phi(e_{1i}) =\sum_{g\in G} h_{i,g} u_g^* \end{equation}
	for $h_{i,g} \in C(X)$.
	Since $\phi(e_{1i})$ is a normalizer, by Corollary \ref{cor:SeminormalizerCoeffs}, the collections $\{\suppo(h_{i,g}):g\in G\}$ and $\{\suppo(  h_{i,g}\circ \alpha_g^{-1}  ):g\in G\}$ are both pairwise disjoint.
	Since $\phi$ is order zero, we have
	\begin{equation} \phi(e_{11}) = (\phi(e_{1i})\phi(e_{1i})^*)^{\frac{1}{2}}=\sum_{g\in G} |h_{i,g}| \end{equation}
	(where the sum is orthogonal and norm-convergent).
	Since this is true for all $i$, it follows that we may find a pairwise disjoint family $(f_t)_{t\in T}$ in $C(X)_+$ (indexed by some countable set $T$) along with a function $s:T\times \{1,\dots,n\} \to G$ such that
	\begin{equation} |h_{i,g}| = \sum_{t\in T: s(t,i)=g} f_t. \end{equation}
	Since the orthogonal sum $\sum_{g\in G} |h_{i,g}| = \sum_{t\in T} f_t$ converges, we must have $\|f_t\| \to 0$ as $t\to\infty$.
	For each $t\in T$ and $i=1,\dots,n$, we have that $f_t=|h_{i,s(t,i)}|$ on $\suppo(f_t)$ (by the orthogonality of the $f_t$), so we may define $\theta_{t,i}:\suppo(f_t) \to \mathbb T$ by
	\begin{equation} \theta_{t,i}(x):=\frac{f_t(x)}{h_{i,s(t,i)}(x)}, \end{equation}
	and this is a continuous function.
	We obtain
	\begin{equation} h_{i,g} = \sum_{t\in T:s(t,i)=g} \bar\theta_{t,i}f_t. \end{equation}
	
	Since $\phi(e_{11})\in C(X)_+$, we have $s(t,1)=e$ and $\theta_{t,1}\equiv 1$ for all $t\in T$.
	We may therefore rewrite
	\begin{equation} \phi(e_{1j}) = \sum_{t\in T} \bar\theta_{t,j}f_tu_{s(t,j)}^* =\sum_{t\in T} u_{s(t,1)}\theta_{t,i}\bar\theta_{t,j}f_tu_{s(t,j)}^*. \end{equation}
	Since $\phi$ is order zero, we also obtain
	\begin{equation} \phi(e_{ij}):=\sum_{t\in T} u_{s(t,i)}\theta_{t,i}\bar\theta_{t,j}f_tu_{s(t,j)}^*. \end{equation}
	It remains only to show that when we set
	\begin{equation} S_t:=\{s(t,1),\dots,s(t,n)\}, \end{equation}
	we have that $((\suppo(f_t),S_t))_{t\in T}$ is a castle.
	
	For this, first since
	\begin{equation} \phi(e_{i1}) = \sum_{t\in T} u_{s(t,i)}\theta_{t,i}f_t = \sum_{t\in T} ((\theta_{t,i}f_t)\circ\alpha_{s(t,i)}^{-1})u_{s(t,i)} \end{equation}
	is a normalizer, it follows from Corollary \ref{cor:SeminormalizerCoeffs} (and the fact that the $f_t$ are orthogonal) that
	\begin{equation} \{\suppo(f_t\circ\alpha_{s(t,i)}^{-1}): t \in T\} = \{\alpha_{s(t,i)}(\suppo(f_t)):t\in T\} \end{equation}
	is pairwise orthogonal.
	Moreover, we compute
	\begin{equation} \phi(e_{ii}) = \sum_{t\in T} u_{s(t,i)}f_tu_{s(t,i)}^* = \sum_{t\in T} f_t\circ\alpha_{s(t,i)}^{-1}, \end{equation}
	so using the orthogonality of the above family,
	\begin{equation} \suppo(\phi(e_{ii})) = \coprod_{t\in T} \alpha_{s(t,i)}(\suppo(f_t)). \end{equation}
	Since $\phi$ is order zero, we know that $\phi(e_{ii})$ and $\phi(e_{jj})$ are orthogonal for all $i\neq j$.
	Consequently, we find that the entire family
	\begin{equation} \{\alpha_{s(t,i)}(\suppo(f_t)):t\in T, i=1,\dots,n\} \end{equation}
	is pairwise disjoint, which means that $((\suppo(f_t),S_t)_{t\in T}$ is a castle.
\end{proof}

\newcommand{\cstar}{$C^*$}

\end{document}